\documentclass[a4paper,11pt] {amsart}
\usepackage{amssymb,latexsym,amsmath}
\usepackage{amsfonts,amsbsy,bm}
\usepackage{color}
\def\a{\alpha}

\def\ga{\gamma}

\def\ba{\beta}

\newcommand{\bprop} {\begin{proposition}}
\newcommand{\eprop} {\end{proposition}}
\newcommand{\btheo} {\begin{theorem}}
\newcommand{\etheo} {\end{theorem}}
\newcommand{\blem} {\begin{lemma}}
\newcommand{\elem} {\end{lemma}}
\newcommand{\bcor} {\begin{corollary}}
\newcommand{\ecor} {\end{corollary}}

\newcommand{\Be}{\begin{equation}}
\newcommand{\Ee}{\end{equation}}
\newcommand{\Bea}{\begin{eqnarray}}
\newcommand{\Eea}{\end{eqnarray}}
\newcommand{\Bes}{\begin{equation*}}
\newcommand{\Ees}{\end{equation*}}
\newcommand{\Beas}{\begin{eqnarray*}}
\newcommand{\Eeas}{\end{eqnarray*}}
\newcommand{\Ba}{\begin{array}}
\newcommand{\Ea}{\end{array}}
\def\R{\mathbb{R}}

\begin{document}
\theoremstyle{plain}% Theorem-like structures
\newtheorem{theorem}{Theorem}[section]
\newtheorem{corollary}[theorem]{Corollary}
\newtheorem{lemma}[theorem]{Lemma}
\newtheorem{proposition}[theorem]{Proposition}

\theoremstyle{definition}
\newtheorem{definition}[theorem]{Definition}
\newtheorem{example}[theorem]{Example}

\theoremstyle{remark}
\newtheorem{remark}[theorem]{Remark}
%%%%%%%%%%%%%%%%%%%%%%%%%%%%%%%%%%%%%%%%%%%%%%%%%%%%%%%%%%%%%%%%%%%%
\title[Hilbert-type and Bergman-type operators]{Boundedness of a family of Hilbert-type operators and of its Bergman-type analogue}
%--------author
\author{Justice S. Bansah and Beno\^it F. Sehba}
\address{Department of Mathematics, University of Ghana,\\ P. O. Box LG 62 Legon, Accra, Ghana}
\email{fjaccobian@gmail.com}
%\author{Beno\^it Florent Sehba}
\address{Department of Mathematics, University of Ghana,\\ P. O. Box LG 62 Legon, Accra, Ghana}
\email{bfsehba@ug.edu.gh}


\maketitle
%-------abstract
\begin{abstract}
In this paper, we first consider boundedness
properties of a family of operators generalizing the Hilbert operator in the upper triangle case. In the diagonal case, we give the exact norm of these operators under some restrictions on the parameters.    We secondly consider boundedness properties of a family of positive Bergman-type operators of the upper-half plane. We give necessary and sufficient conditions on the parameters under which these operators are bounded in the upper triangle case.
\end{abstract}
%%% ----------------------------------------------------------------------
%\maketitle
%%% ----------------------------------------------------------------------
\section{Introduction}
Let $1\le p<\infty$, and $a>-1$. We write $L_a^p((0,\infty))$ or simply $L_a^p$ for the Lebesgue space $L^p\left((0,\infty), y^ady\right)$. When $a=0$, we simply write $L^p$ for the corresponding space. We use the notion 
$$\|f\|_{p,a}:=\|f\|_{L_a^p}:=\left(\int_0^\infty |f(x)|^px^adx\right)^{\frac{1}{p}}$$
and when $a=0$, we simply write $\|f\|_{p}$ for $\|f\|_{p,0}$.
\vskip .2cm
We recall that the Hilbert operator is defined by $$Hf(x):=\int_0^\infty\frac{f(y)}{x+y}\,dy.$$  It is well known that the operator $H$ is bounded on $L^p((0,\infty))$ for $1<p<\infty$ and that its norm is given by $$\|H\|_{L^p\rightarrow L^p}=\frac{\pi}{\sin\left(\frac{\pi}{p}\right)}$$
(see \cite{Hardy, HLP}). For more on the Hilbert operator, its boundedness, some generalizations and applications, we refer to the following and the references therein \cite{MPF, Byang1, Byang2, Byang3}.
\vskip .2cm
Let $\alpha, \beta, \gamma$ be real parameters. Consider the family of operators $H_{\alpha,\beta,\gamma}$ defined for compactly supported functions by
\begin{equation}\label{eq:hilbertgenedef}
H_{\alpha,\beta,\gamma}f(x):=x^\alpha\int_0^\infty\frac{f(y)}{(x+y)^\gamma}y^\beta \mathrm{d}y.
\end{equation}
The above operators clearly generalize the Hilbert operator as $H=H_{0,0,1}$.

 In the first part of this note, we consider the continuity properties of the operators $H_{\alpha,\beta,\gamma}$ from $L_a^p$ to $L_b^q$, with $1\le p\le q<\infty$. That is we give the relations between $p,q,a,b,\alpha,\beta,\gamma$ under which these operators are bounded. Restricting ourself to the case $\gamma=\alpha+\beta+1$ and $p=q<\infty$, we give the exact norm of $H_{\alpha,\beta,\gamma}$, extending the results of \cite{Hardy, HLP, Byang1}.
\vskip .2cm
We  recall that the upper-half plane is $\mathbb{R}_+^2:=\{x+iy\in \mathbb{C}:x\in \mathbb{R}\,\,\,\textrm{and}\,\,\,y>0\}$. Given $1\le p,q\le \infty$ and $\nu>-1$, the mixed
norm Lebesgue space $L_\nu^{p,q}(\mathbb{R}_+^2)$ is defined by the
integrability condition \Be
\|f\|_{p,q,\nu}=||f||_{L_\nu^{p,q}}:=\left[\int_0^\infty\left(\int_{\R}|f(x+iy)|^{p}\mathrm{d}x\right)^{\frac{q}{p}}y^\nu \mathrm{d}y\right]^{\frac{1}{q}}<\infty\,\,\,\textrm{if}\,\,\, 1\le p,q< \infty
\Ee 
and 
\Be
\|f\|_{p,\infty}^p=||f||_{L^{p,\infty}}^p:=\sup_{0<y<\infty}\int_{\R}|f(x+iy)|^{p}\mathrm{d}x <\infty\,\,\,\textrm{if}\,\,\, 1\le p<\infty\,\,\,\textrm{and}\,\,\,q=\infty.
\Ee 
For $1\le p,q<\infty$, the mixed norm weighted Bergman space $A_\nu^{p,q}(\mathbb{R}_+^2)$ is then the closed subspace of $L_\nu^{p,q}(\mathbb{R}_+^2)$ consisting of
holomorphic functions on $\mathbb{R}_+^2$.  When $p=q $ we shall simply write $A_\nu^{p,p}(\mathbb{R}_+^2)=A_\nu^p(\mathbb{R}_+^2)$. The unweighted Bergman space $A^p$ corresponds to the case $\nu=0$.

Recall also that the weighted Bergman projection $P_\nu$ is the orthogonal
projection from the Hilbert space $L_{\nu}^2(\mathbb{R}_+^2)$ onto its
closed subspace $A_{\nu}^2(\mathbb{R}_+^2)$ and it is given by the integral
formula \Be P_{\nu}f(z)=c_\nu \int_{\mathbb{R}_+^2}\frac{f(w)}{(z-\bar{w})^{2+\nu}}dV_\nu(w)\Ee where we used the notation $dV_\nu(z)=y^\nu dV(z)=y^\nu dxdy$, $z=x+iy$; $c_\nu=\frac{2^\nu}{\pi}(\nu+1)e^{-i\nu \frac{\pi}{2}}$.

It is well known that for any $f\in A_{\nu}^2(\mathbb{R}_+^2)$, \Be f(z)=c_\nu \int_{\mathbb{R}_+^2}\frac{f(w)}{(z-\bar{w})^{2+\nu}}dV_\nu(w). \Ee

It is easy to see that the Bergman projection is bounded on $L_\nu^{p,q}(\mathbb{R}_+^2)$ whenever $q>1$ (see \cite{BBGNPR}). 

Our second interest in this paper is the upper-diagonal-boundedness of a family of
operators generalizing the Bergman projection. This family is
given by the integral operators $T=T_{\a,\ba,\ga}$ and
$T^{+}=T_{\a,\ba,\ga}^+$ defined for functions in $C_c^{\infty}(\mathbb{R}_+^2)$ by the
formulas

$$Tf(z)=(\Im z)^\alpha \int_{\mathbb{R}_+^2}\frac{f(w)}{(z-\bar{w})^{1+\gamma}}(\Im w)^\beta dV(w),$$
and
$$T^{+}f(z)=(\Im z)^\alpha \int_{\mathbb{R}_+^2}\frac{f(w)}{|z-\bar{w}|^{1+\gamma}}(\Im w)^\beta dV(w).$$
Let us remark that the boundedness of $T^+$ on $L_\nu^{p,q}(\mathbb{R}_+^2)$
implies the boundedness of $T$.

The boundedness of this family of operators on $L_\nu^{p,q}(\mathbb{R}_+^2)$ for $1<p,q<\infty$ is just a particular case of \cite{Sehba}. Here, we consider the problem of the boundedness of the operators  $T^+$ from $L_\nu^{p,q}(\mathbb{R}_+^2)$ to $L_\mu^{p,r}(\mathbb{R}_+^2)$, with $1\le p<\infty$ and $1\le q\le r<\infty$. 
\vskip .2cm
As we will see, the study of the boundedness of the operators $T_{\alpha,\beta,\gamma}^+$ can be related to the boundedness of the operators $H_{\alpha,\beta,\gamma}$, providing another motivation for the study of the general Hilbert operators considered here. The Bergman projection is just a particular case of the operators $T_{\alpha, \beta,\gamma}$ and his boundedness is useful in some other questions as the characterization of the dual spaces of Bergman spaces and their atomic decomposition (see for example \cite{BBGNPR}). 

\section{Statement of the results}
\subsection{Hilbert-type operators}
The following result provide relations between $p,q,a,b,\alpha,\beta,\gamma$ under which the operators $H_{\alpha, \beta \gamma}$ are bounded. 
\begin{theorem}\label{thm:main1}
Suppose $a,b>-1$ and $1\le p\le q< \infty$. Then the
following conditions are equivalent:

\begin{itemize}
\item[(a)]

The operator $H_{\a,\ba,\ga}$ is bounded from
$L_{a}^{p}((0,\infty))$ into $L_{b}^{q}((0,\infty))$

\item[(b)]

The parameters satisfy 
\begin{equation}\label{eq:relationalphabetagamma1}
\gamma =\a+\ba+1-\frac{a+1}{p}+\frac{b+1}{q}
\end{equation} 
and
\begin{equation}\label{eq:condineq11}
-p(\gamma-\beta-1)<a+1<p(\beta+1).
\end{equation}
\end{itemize}
\end{theorem}
%Let us remark that in the above theorem, under the relation %(\ref{eq:relationalphabetagamma1}), the condition (\ref{eq:condineq11}) is equivalent to the %following condition
%\begin{equation}\label{eq:condineq2}
%-q\alpha<b+1<q(\gamma-\alpha).
%\end{equation}
We also have the following second result.
\begin{theorem}\label{thm:main2pinfty}
Suppose $a>-1$ and $1<p< \infty$. Then the
following conditions are equivalent:

\begin{itemize}
\item[(a)]

The operator $H_{\a,\ba,\ga}$ is bounded from
$L_{a}^{p}((0,\infty))$ into $L^\infty((0,\infty))$.

\item[(b)]

The parameters satisfy 
\begin{equation}
\gamma =\a+\ba+1-\frac{a+1}{p}
\end{equation} 
and
\begin{equation}\label{eq:condineq1}
\alpha>0\,\,\,\textrm{and}\,\,\,a+1<p(\beta+1).
\end{equation}
\end{itemize}
\end{theorem}
In the diagonal case, we have the following endpoint result.
\begin{theorem}\label{thm:main3inftyinfty}

The operator $H_{\a,\ba,\ga}$ is bounded on
$L^\infty((0,\infty))$ if and only if
$\alpha>0$, $\beta>-1$ and $\gamma=\alpha+\beta+1$. Moreover,
$$\|H_{\alpha,\beta,\gamma}\|_{L^\infty \rightarrow L^\infty}=B\left(\beta+1,\alpha\right).$$

\end{theorem}
In the above theorem and all over this section, $B(\cdot, \cdot)$ is the $\beta$-function defined in the next section. Restricting ourself to the case $\gamma=\alpha+\beta+1$ and $p=q<\infty$, we obtain the exact norm of the corresponding operators $H_{\alpha,\beta,\gamma}$.
\begin{corollary}\label{cor:maincor1}
Let $a>-1$ and $1\le p<\infty$. Assume that $-p\alpha<a+1<p(\beta+1)$ and $\gamma=\alpha+\beta+1$. Then $$\|H_{\alpha,\beta,\gamma}\|_{L_a^p\rightarrow L_a^p}=B\left(\beta+1-\frac{a+1}{p},\alpha+\frac{a+1}{p}\right).$$
\end{corollary}
\subsection{Bergman-type operators}
Here are our results on the boundedness of the operators $T^+$ from $L_{a}^{p,q}(\mathbb{R}_+^2)$ into $L_{b}^{p,r}(\mathbb{R}_+^2)$.
\subsubsection{The case $1<p,q<\infty$}

We have the following result.
\begin{theorem}\label{thm:main2}
Suppose $a,b>-1$, $1< p<\infty$, and $1< q\le r< \infty$. Then the
following conditions are equivalent:

\begin{itemize}
\item[(a)]

The operator $T^+$ is bounded from
$L_{a}^{p,q}(\mathbb{R}_+^2)$ into $L_{b}^{p,r}(\mathbb{R}_+^2)$

\item[(b)]

The parameters satisfy 
\begin{equation}\label{eq:relationalphabetagamma}
\gamma =\a+\ba+1-\frac{a+1}{q}+\frac{b+1}{r}
\end{equation} 
and
\begin{equation}\label{eq:condineq1}
-q(\gamma-\beta-1)<a+1<q(\beta+1).
\end{equation}
\end{itemize}
\end{theorem}
We also obtain the following.
\begin{theorem}\label{thm:main3}
Suppose $a>-1$, $1< p<\infty$ and $1<q<\infty$. Then the
following conditions are equivalent:

\begin{itemize}
\item[(a)]

The operator $T^+$ is bounded from
$L_{a}^{p,q}(\mathbb{R}_+^2)$ into $L^{p,\infty}(\mathbb{R}_+^2)$

\item[(b)]

The parameters satisfy 
\begin{equation}\label{eq:main31}
\gamma =\a+\ba+1-\frac{a+1}{q}
\end{equation} 
and
\begin{equation}\label{eq:main32}
\alpha>0\,\,\,\textrm{and}\,\,\, a+1<q(\beta+1).
\end{equation}
\end{itemize}
\end{theorem}
\subsubsection{The case $1<p<\infty$ and $1=q\le r\le \infty$}

We have the following result.
\begin{theorem}\label{thm:main4}
Let $1< p,r<\infty$ and let $b>-1$. Then the
following conditions are equivalent:

\begin{itemize}
\item[(a)]

The operator $T^+$ is bounded from
$L^{p,1}(\mathbb{R}_+^2)$ into $L_b^{p,r}(\mathbb{R}_+^2)$.

\item[(b)]

The parameters satisfy 
\begin{equation}\label{eq:main41}
\gamma =\a+\ba+\frac{b+1}{r}
\end{equation} 
and
\begin{equation}\label{eq:main42}
\gamma>\beta>0.
\end{equation}
\end{itemize}
\end{theorem}
Note that Theorem \ref{thm:main4} is the dual version of Theorem \ref{thm:main3}. The limit case is the following.
\begin{theorem}\label{thm:main5}
Let $1< p<\infty$. Then the
following conditions are equivalent:

\begin{itemize}
\item[(a)]

The operator $T^+$ is bounded from
$L^{p,1}(\mathbb{R}_+^2)$ into $L^{p,\infty}(\mathbb{R}_+^2)$.

\item[(b)]

The parameters satisfy 
\begin{equation}\label{eq:main51}
\gamma =\a+\ba
\end{equation} 
and
\begin{equation}\label{eq:main52}
\alpha,\beta>0.
\end{equation}
\end{itemize}
\end{theorem}
We also obtain the following.
\begin{theorem}\label{thm:main6}
Let $1< p<\infty$. Then the
following conditions are equivalent:

\begin{itemize}
\item[(a)]

The operator $T^+$ is bounded on
$L^{p,1}(\mathbb{R}_+^2)$.

\item[(b)]

The parameters satisfy 
\begin{equation}\label{eq:main61}
\gamma =\a+\ba+1
\end{equation} 
and
\begin{equation}\label{eq:main62}
\alpha>-1\,\,\,\textrm{and}\,\,\,\beta>0.
\end{equation}
\end{itemize}
\end{theorem}
%\subsubsection{The case $q=r=\infty$}
%Taking $q=r=\infty$, we obtain the following.
The dual version of the above result is the following.
\begin{theorem}\label{thm:main7}
Let $1< p<\infty$. Then the
following conditions are equivalent:

\begin{itemize}
\item[(a)]

The operator $T^+$ is bounded on
$L^{p,\infty}(\mathbb{R}_+^2)$.

\item[(b)]

The parameters satisfy 
\begin{equation}
\gamma =\a+\ba+1
\end{equation} 
and
\begin{equation}\label{eq:main72}
\alpha>0\,\,\,\textrm{and}\,\,\, \beta>-1.
\end{equation}
\end{itemize}
\end{theorem}
\subsubsection{Diagonal limit cases}
Finally, we obtain the following two endpoint results. 
\begin{theorem}\label{thm:main8} The operator $T_{\a,\ba,\ga}^{+}$ is bounded on
$L^{\infty}(\mathbb{R}_+^2)$ if and only if $\a >0$,  $\ba
> -1$ and $\ga=\a + \ba + 1$. Moreover,
$$\|T_{\a,\ba,\ga}^{+}\|_{L^\infty (\mathbb{R}_+^2)\rightarrow L^\infty (\mathbb{R}_+^2)}=B(\frac{1}{2},\frac{\gamma}{2})B(\beta+1,\alpha).$$
\end{theorem}

\begin{theorem}\label{thm:main9} Let $a>-1$. Then the operator $T_{\a,\ba,\ga}^{+}$ is bounded on
$L_a^{1}(\mathbb{R}_+^2)$ if and only if $\ga=\a + \ba + 1$ and $-\alpha<a+1<\beta+1$. Moreover,
$$\|T_{\a,\ba,\ga}^{+}\|_{L_a^1 (\mathbb{R}_+^2)\rightarrow L_a^1 (\mathbb{R}_+^2)}=B(\frac{1}{2},\frac{\gamma}{2})B(\beta-a,\alpha+a+1).$$
\end{theorem}
\vskip .3cm
In the next section we provide some useful tools needed in our proofs. The proofs of our results are given in section 4 and section 5. In the last section, we discuss the boundedness of the operators $T_{\alpha,\beta, \gamma}$ with application to the Bergman projection.
\vskip .2cm
As usual, given two positive quantities $A$ and $B$, the notation $A\lesssim B$ (or $B\gtrsim A$) means that there is universal positive constant $C$ such that $A\le CB$. When $A\lesssim B$ and $B\lesssim A$, we say $A$ and $B$ are equivalent and write $A\simeq B$. 
\section{Some Useful Results}
\subsection{Integrability of some positive kernel functions}
We shall use the following form of the $\beta-$function: $$B(m, n)= B(n,m) = \int_0^{\infty}\!\!\!\frac{u^{m-1}}{(1+u)^{m+n}}\mathrm{d}u\qquad\mbox{where}\quad m,n>0.$$
More generally, we will be using the following which is easy to check.
\begin{lemma}\label{lem:betafunctionconditions}
Let $\alpha, \beta$ be a real numbers, and $t>0$ be fixed. Then the integral $$I(t)=\int_0^\infty \frac{y^\alpha}{(t+y)^\beta}\mathrm{d}y$$
converges if and only if $\alpha>-1$ and $\beta-\alpha>1$. In this case, $$I(y)=B(\alpha+1, \beta-\alpha-1)t^{-\beta+\alpha+1}.$$
\end{lemma}
We will need the following integrability conditions of the kernel function.
\begin{lemma}\label{lem:integkernel} Let $\a$ be real. Then
\begin{itemize}
\item[(1)] for $y>0$ fixed, the integral
$$J_{\a}(y)=\int_{\R}\frac{dx}{|x+iy|^\alpha}
$$ converges if and only if $\a > 1.$ In this case,
$$J_{\a}(y)=B(\frac{1}{2}, \frac{\alpha-1}{2})y^{1-\alpha};$$
\item[(2)] the function
$f(z)=\left(\frac{z+it}{i}\right)^{-\alpha}$, with $t>0$, belongs to
$L_{\nu}^{p,q}(\mathbb{R}_+^2)$, if and only if $\nu>-1$ and $\a >
\frac{1}{p} + \frac{\nu+1}{q}$. In this
case,$$||f||_{p,q,\nu}^q=C_{\a,p,q}t^{-q\alpha+\frac{q}{p}+\nu+1}$$
where $C_{\a,p,q}=[B(\frac{1}{2}, \frac{p\alpha-1}{2})]^{\frac{q}{p}}B(\nu+1, q\alpha-\frac{q}{p}-\nu-1)$.
\end{itemize}
\end{lemma}
\begin{proof} 
$(1)$ Note that $$J_{\a}(y)=2\int_0^\infty \frac{dx}{(x^2+y^2)^{\frac{\alpha}{2}}}.$$
If $\alpha\le 1$, then as $\int_y^\infty x^{-\alpha}dx$ which is smaller than $J_{\a}(y)$ does not converge, neither  does $J_{\a}(y)$. 

Assuming that $\alpha>1$, the convergence and the value of $J_{\a}(y)$ follows from an easy change of variables and Lemma \ref{lem:betafunctionconditions}.

The proof of assertion $(2)$ follows from assertion $(1)$ and Lemma \ref{lem:betafunctionconditions}.
\end{proof}

\subsection{Schur-type tests}
The following is a generalization of the Schur's test and it is due to G.O. Okikiolu \cite{Oki}. Our statement is a bit different from \cite{Oki}. We also provide a different proof.
\begin{lemma}\label{lem:okikiolu}
%Let $p,r,\alpha_1,\alpha_2$ be positive numbers such that $1<p\le r, p$ and $p'$ being conjugate exponents, $\frac{\alpha_1}{p'} + \frac{\alpha_2}{r} = 1.$ Let also $K(x,y)$ be a complex-value function measurable on $X\times Y$ and suppose there exist measurable functions $\phi_1:X\rightarrow (0,\infty),\quad\phi_2:Y\rightarrow (0,\infty)$ and nonnegative constants $M_1,M_2$ such that
%\begin{eqnarray}
%\label{e1} 
%\int_X\left|K(x,y)\right|^{\alpha_1}\phi_1^{p'}(y)\mathrm{d}\mu(y) &\le & M_1^{p'}\phi_2^{p'}(x)\qquad\mbox{a.e on}\quad Y\quad\mbox{and}\\ 
%\label{e2}
%\int_Y\left|K(x,y)\right|^{\alpha_2}\phi_2^r(x)\mathrm{d}\nu(x) &\le  &M_2^r\phi_1^r(y)\qquad\mbox{a.e on}\quad X.
%\end{eqnarray} 
%If $T$ is given by $$Tf(x)=\int_X\!\!\!f(y)K(x,y)\mathrm{d}\mu(y)$$ where $f\in L^p(X,\mathrm{d}\mu),$ then $T:L^p(X,\mathrm{d}\mu)\longrightarrow L^r(Y,\mathrm{d}\nu)$ is bounded and for each $f\in L^p(X,\mathrm{d}\mu)$, $$\left\|Tf\right\|_{r,\nu}\le M_1M_2\|f\|_{p,\mu}$$
%\end{lemma}
Let $p,r,q$ be positive numbers such that $1<p\le r$ and $\frac{1}{p} + \frac{1}{q} = 1.$ Let $K(x,y)$ be a complex-value function measurable on $X\times Y$ and suppose there exist $0<t\le 1$, measurable functions $\phi_1:X\rightarrow (0,\infty),\quad\phi_2:Y\rightarrow (0,\infty)$ and nonnegative constants $M_1,M_2$ such that
\begin{eqnarray}
\label{ee1} 
\int_X\left|K(x,y)\right|^{tq}\phi_1^{q}(y)\mathrm{d}\mu(y) &\le & M_1^{q}\phi_2^{q}(x)\qquad\mbox{a.e on}\quad Y\quad\mbox{and}\\ 
\label{ee2}
\int_Y\left|K(x,y)\right|^{(1-t)r}\phi_2^r(x)\mathrm{d}\nu(x) &\le  &M_2^r\phi_1^r(y)\qquad\mbox{a.e on}\quad X.
\end{eqnarray} 
If $T$ is given by $$Tf(x)=\int_X\!\!\!f(y)K(x,y)\mathrm{d}\mu(y)$$ where $f\in L^p(X,\mathrm{d}\mu),$ then $T:L^p(X,\mathrm{d}\mu)\longrightarrow L^r(Y,\mathrm{d}\nu)$ is bounded and for each $f\in L^p(X,\mathrm{d}\mu)$, $$\left\|Tf\right\|_{L^r(Y,\mathrm{d}\nu)}\le M_1M_2\|f\|_{L^r(X,\mathrm{d}\mu)}.$$
\end{lemma}

\begin{proof}
Using H\"{o}lder's inequality and \eqref{ee1}, we obtain that 
\begin{eqnarray*}
|Tf(x)|& \le & \int_X\!\!\!|f(y)||K(x,y)|\mathrm{d}\mu(y)\\
	& = & \int_X\!\!\left[|K(x,y)|^t\phi_1(y)\right]\left[|K(x,y)^{1-t}\phi_1^{-1}(y)|f(y)|\right]\mathrm{d}\mu(y)\\
	& \le & \left[\int_X\!\!|K(x,y)|^{tq}\phi_1^q(y)\mathrm{d}\mu(y)\right]^{\frac{1}{q}}\left[\int_X\!\!|K(x,y)|^{(1-t)p}\phi_1^{-p}(y)|f(y)|^p\mathrm{d}\mu(y)\right]^{\frac{1}{p}}\\
	& \le & M_1\phi_2(x)\left[\int_X\!\!|K(x,y)|^{(1-t)p}\phi_1^{-p}(y)|f(y)|^p\mathrm{d}\mu(y)\right]^{\frac{1}{p}}.
\end{eqnarray*}
Using Minkowski's inequality for double integrals and \eqref{ee2}, we obtain
\begin{eqnarray*}
\left\|Tf\right\|_{L^r(Y,\mathrm{d}\nu)} &=& \left(\int_Y\!\!|Tf(x)|^r\mathrm{d}\nu(x)\right)^{\frac{1}{r}}\\ & \le & \left(\int_Y\!\!M_1^r\phi_2^r(x)\left[\int_X\!\!|K(x,y)|^{(1-t)p}\phi_1^{-p}(y)|f(y)|^p\mathrm{d}\mu(y)\right]^{\frac{r}{p}} \mathrm{d}\nu(x)\right)^{\frac{1}{r}}\\
& = & M_1\left(\int_Y\!\!\phi_2^r(x)\left[\int_X\!\!|K(x,y)|^{(1-t)p}\phi_1^{-p}(y)|f(y)|^p\mathrm{d}\mu(y)\right]^{\frac{r}{p}} \mathrm{d}\nu(x)\right)^{\frac{p}{r}\times\frac{1}{p}}\\
& \le &M_1\left(\int_X\!\!\left[\int_Y\!\!|K(x,y)|^{(1-t)r}\phi_2^r(x)\mathrm{d}\nu(x)\right]^{\frac{p}{r}}\phi_1^{-p}(y)|f(y)|^p\mathrm{d}\mu(y) \right)^{\frac{1}{p}}\\
& = & M_1\left(\int_X\!\!\phi_1^{-p}(y)M_2^p\phi_1^p(y)|f(y)|^p\mathrm{d}\mu(y) \right)^{\frac{1}{p}}\\
& = & M_1M_2\left(\int_X\!\!|f(y)|^p\mathrm{d}\mu(y)\right)^{\frac{1}{p}}\\
& = & M_1M_2\|f\|_{L^p(X,\mathrm{d}\mu)}.  
\end{eqnarray*} 
\end{proof}
When $p=r$, take $t=\frac{1}{q}$ to obtain the classical Schur's test. 

The following limit case of Okikiolu result is proved in \cite{Zhao}.
\begin{lemma}\label{lem:okikiolulimitcase}
Let $\mu$ and $\nu$ be positive measures on the space $X$ and let $K(x, y)$ be non-negative measurable functions on $X\times Y.$ Let $T$ be the integral operator with kernel $K(x,y)$ defined by $$Tf(x) = \int_X\!\! f(y)K(x, y)\mathrm{d}\mu(y).$$ Suppose $1=p\le q<\infty.$ Let $\gamma$ and $\delta$ be two real numbers such that $\gamma + \delta = 1.$ If there exist positive functions $h_1$ and $h_2$ with positive constants $C_1$ and $C_2$ such that $$\mbox{ess}\sup_{y\in Y}h_1(y)K(x,y)^{\gamma}\le C_1h_2(x)\quad\mbox{for almost all } x\in X\quad\mbox{and}$$ $$\int_X\!\!h_2(x)^qK(x,y)^{\delta q}\mathrm{d}\nu\le C_2h_1(y)^q\quad\mbox{for almost all } y\in Y,$$ then $T$ is bounded from $L^1(X,\mathrm{d}\nu)$ into $L^q(X,\mathrm{d}\nu)$ and the norm of this operator does not exceed $C_1C_2^{\frac{1}{q}}.$
\end{lemma}
The above lemma has a very simple formulation when $p=q=1$, and the exact norm of the operator is also obtained (see \cite{Marko}).
\begin{lemma}\label{lem:okikiolulimitcasesharp}
Let $\mu$ be a positive measure on the space $X$ and let $K(x, y)$ be non-negative measurable functions on $X\times X$. Let $T$ be the integral operator with kernel $K(x,y)$ defined by $$Tf(x) = \int_X\!\! f(y)K(x, y)\mathrm{d}\mu(y).$$ Then $T$ is bounded on $L_\mu^1(X):=L^1(X,d\mu)$ if and only if  $$\sup_{y\in X}\int_X\,K(x,y)\mathrm{d}\mu(x)<\infty.$$ 
In this case, $$\|T\|_{L_\mu^1\rightarrow L_\mu^1}=\sup_{y\in X}\int_X\,K(x,y)\mathrm{d}\mu(x).$$
\end{lemma}
The dual version of the last lemma is the following.
\begin{lemma}\label{lem:okikiolulimitcaseinfty}
Let $\mu$ be a positive measure on the space $X$ and let $K(x, y)$ be non-negative measurable functions on $X\times X$. Let $T$ be the integral operator with kernel $K(x,y)$ defined by $$Tf(x) = \int_X\!\! f(y)K(x, y)\mathrm{d}\mu(y).$$ Then $T$ is bounded on $L^\infty (X)$ if and only if  $$\sup_{x\in X}\int_X\,K(x,y)\mathrm{d}\mu(y)<\infty.$$ 
In this case, $$\|T\|_{L^\infty \rightarrow L^\infty}=\sup_{x\in X}\int_X\,K(x,y)\mathrm{d}\mu(y).$$
\end{lemma}

\section{Boundedness of a family of Hilbert-type Operators.}
In this section, we prove some necessary and sufficient conditions for the boundedness of the operator $H_{\alpha, \beta,\gamma}$ from $L^p((0,\infty), y^a\mathrm{d}y)$ to $L^q((0,\infty), y^b\mathrm{d}y)$. We recall that the operator $H_{\alpha, \beta,\gamma}$ is defined as $$H_{\alpha,\beta,\gamma}f(x) = x^{\alpha}\int_0^{\infty}\!\!\frac{f(y)}{(x+y)^{\gamma}}y^{\beta}\mathrm{d}y.$$

\subsection{Necessity for Boundedness of $H_{\alpha,\beta, \gamma}$}
Let us start by the following lemma.
\begin{lemma}\label{lem:necessdirect}
Let $1\le p\le q<\infty$ and $a,b>-1.$ Assume $H_{\alpha,\beta,\gamma}$ is bounded from $ L^p((0,\infty), y^a\mathrm{d}y)$ to $ L^q((0,\infty), y^b\mathrm{d}y)$. Then the parameters satisfy   $$\gamma=\alpha+\beta+1+\left(\frac{b+1}{q} - \frac{a+1}{p}\right)$$
and 
$$-q\alpha<b+1<q(\gamma-\alpha).$$
\end{lemma}
\begin{proof}
Let $R>0$ and define $f_R(x):=f(Rx).$ Then it is easy to see that if $f\in  L^p((0,\infty), y^a\mathrm{d}y)$, then $f_R\in  L^p((0,\infty), y^a\mathrm{d}y)$ and $\|f_R\|^p_{p,a}=R^{-a-1}\|f\|^p_{p,a}.$ From the definition of $H_{\alpha,\beta,\gamma}$ and some easy change of variables, we obtain
%By definition, $H_{\alpha,\beta,\gamma}f_R(x)=H_{\alpha,\beta,\gamma}f(Rx).$ i.e
\begin{eqnarray*}
H_{\alpha,\beta,\gamma}f_R(x) &= & x^{\alpha}\int_0^{\infty}\!\!\frac{f_R(y)}{(x+y)^{\gamma}}y^{\beta}\mathrm{d}y\\
& = & x^{\alpha}\int_0^{\infty}\!\!\frac{f(Ry)}{(x+y)^{\gamma}}y^{\beta}\mathrm{d}y\\
& = & R^{\gamma-\beta-1}x^{\alpha}\int_0^{\infty}\!\!\frac{f(u)}{(Rx+u)^{\gamma}}u^{\beta}\mathrm{d}u,\quad\mbox{(letting $u=Ry$)}\\
& = & R^{\gamma-\beta-\alpha-1}(Rx)^{\alpha}\int_0^{\infty}\!\!\frac{f(u)}{(Rx+u)^{\gamma}}u^{\beta}\mathrm{d}u\\
& = & R^{\gamma-\beta-\alpha-1}H_{\alpha,\beta,\gamma}f(Rx).
\end{eqnarray*}
Therefore,
\begin{eqnarray*}
\|H_{\alpha,\beta,\gamma}f_R\|^q_{q,b} & = & R^{q(\gamma-\alpha-\beta-1)}\int_0^{\infty}\!\!|H_{\alpha,\beta,\gamma}f(Rx)|^qx^b\mathrm{d}x\\
& = & R^{q(\gamma-\alpha-\beta-1)-b-1}\int_0^{\infty}\!\!|H_{\alpha,\beta,\gamma}f(u)|^qu^b\mathrm{d}u\quad\mbox{(putting $u=Rx$)}\\
& = & R^{q(\gamma-\alpha-\beta-1)-b-1}\|H_{\alpha,\beta,\gamma}f\|^q_{q,b}.
\end{eqnarray*}
Now, that there exists a constant $C>0$ such that for any $f\in L_a^p((0,\infty))$, $\|H_{\alpha,\beta,\gamma}f_R\|_{q,b}\le C\|f_R\|_{p,a}$ is equivalent to $$R^{\gamma-\alpha-\beta-1-\frac{b+1}{q}}\|H_{\alpha,\beta,\gamma}f\|_{q,b}\le CR^{-\frac{a+1}{p}}\|f\|_{p,a}$$ which is the same as $$R^{\gamma-\alpha-\beta-1-\frac{b+1}{q}+\frac{a+1}{p}}\|H_{\alpha,\beta,\gamma}f\|_{q,b}\le C\|f\|_{p,a}$$ for any $f\in L_a^p((0,\infty))$. That the latter holds for any $f\in L_a^p((0,\infty))$ and any $R>0$ is only possible if $\gamma-\alpha-\beta-1-\frac{b+1}{q}+\frac{a+1}{p}=0.$

To check the other condition, we set $f(x)=\chi_{[1,2]}(x)$. Then one easily obtains $$H_{\alpha,\beta,\gamma}f(x)=x^\alpha \int_0^\infty \frac{f(y)}{(x+y)^\gamma}y^\beta \mathrm{d}y \simeq \frac{x^\alpha}{(1+x)^\gamma}.$$
It follows from our assumption on the operator $H_{\alpha,\beta,\gamma}$ that
$$\int_0^\infty \frac{x^{q\alpha}}{(1+x)^{q\gamma}} x^b\mathrm{d}x \simeq \|H_{\alpha,\beta,\gamma}f\|_{q,b}^q \lesssim \|f\|_{p,a}^q\simeq 1.$$
It follows from Lemma \ref{lem:betafunctionconditions} that we should have
$\alpha q+b+1>0$ and $\gamma q-\alpha q-b-1>0$; that is 
$-q\alpha <b+1<q(\gamma-\alpha)$. The proof is complete.
\end{proof}
We also have the following.
\begin{lemma}\label{lem:necessadjoint}
Let $1< p\le q<\infty$ and $a,b>-1.$ Assume $H_{\alpha,\beta,\gamma}$ is bounded from $ L^p((0,\infty), y^a\mathrm{d}y)$ to $ L^q((0,\infty), y^b\mathrm{d}y)$. Then the parameters satisfy   $$\gamma=\alpha+\beta+1+\left(\frac{b+1}{q} - \frac{a+1}{p}\right)$$
and 
$$-p(\gamma-\beta-1)<a+1<p(\beta+1).$$
\end{lemma}
\begin{proof}
The necessity of the relation $\gamma=\alpha+\beta+1+\left(\frac{b+1}{q} - \frac{a+1}{p}\right)$ is already proved in the previous lemma.
To check the other condition, we note that as $H_{\alpha,\beta,\gamma}$ is bounded from $L_a^p$ to $L_b^q$, its adjoint $H_{\alpha,\beta,\gamma}^*$ is also bounded from $L_b^{q'}$ to $L_a^{p'}$, $\frac{1}{p}+\frac{1}{p'}=1=\frac{1}{q}+\frac{1}{q'}$. One easily obtain that
$$H_{\alpha,\beta,\gamma}^*f(y)=y^{\beta-a}\int_0^\infty \frac{f(x)}{(x+y)^\gamma}x^{\alpha+b} \mathrm{d}x.$$

Let us take again $f(y)=\chi_{[1,2]}(y)$. Then one easily obtain that $$H_{\alpha,\beta,\gamma}^*f(y) \simeq \frac{y^{\beta-a}}{(1+y)^\gamma}.$$
It follows from our assumption on the operator $H_{\alpha,\beta,\gamma}^*$ that
$$\int_0^\infty \frac{y^{p'(\beta-a)+a}}{(1+y)^{p'\gamma}} \mathrm{d}y \simeq \|H_{\alpha,\beta,\gamma}^*f\|_{p',a}^{p'} \lesssim \|f\|_{q',b}^{p'}\simeq 1.$$
It follows from Lemma \ref{lem:betafunctionconditions} that we should have
$p'(\beta-a)+a+1>0$ and $\gamma p'-p'(\beta-a)-a-1>0$; that is 
$-p'(\beta-a)<a+1<p'(\gamma-\beta+a)$ or equivalently $-p(\gamma-\beta-1)<a+1<p(\beta+1)$. The proof is complete.
\end{proof}
To complete this part, let us observe the following. 
\begin{lemma}\label{lem:equivcond}
Let $1\le p\le q<\infty$ and $a,b>-1.$ Assume that there are real numbers $\alpha, \beta$ and $\gamma$ such that $$\gamma=\alpha+\beta+1+\left(\frac{b+1}{q} - \frac{a+1}{p}\right).$$
Then the condition
$$-q\alpha<b+1<q(\gamma-\alpha)$$
is equivalent to
$$-p(\gamma-\beta-1)<a+1<p(\beta+1).$$
\end{lemma}
\subsection{Sufficiency for Boundedness of $H_{\alpha,\beta, \gamma}$}
We start with the case $p>1$. We have the following lemma.
\begin{lemma}\label{lem:suffnotlimit}
Let $1<p\le q<\infty,$ $\alpha,\beta,\gamma \in\mathbb{R};$ and $a,b>-1.$ Assume that $$\gamma = \alpha+\beta+1-\frac{a+1}{p} + \frac{b+1}{q}$$ and that $$-p(\gamma-\beta-1)<a+1<p(\beta + 1).$$ Then the operator $H_{\alpha, \beta,\gamma}$ is bounded from $ L^p((0,\infty), y^a\mathrm{d}y)$ to $ L^q((0,\infty), y^b\mathrm{d}y).$
\end{lemma}
\begin{proof}
We first observe that as $\gamma = \alpha+\beta+1-\frac{a+1}{p} + \frac{b+1}{q},$ the condition $-p(\gamma-\beta-1)<a+1<p(\beta + 1)$ is equivalent to $-\alpha q<b+1<q(\gamma - \alpha).$ Let us put $\omega = \alpha+\beta-\gamma-a$ and observe that $$\omega = \alpha+\beta-\gamma-a = -\left(\frac{a+1}{p'}+\frac{b+1}{q} \right)<0.$$ Now, we observe that $a+1<p(\beta+1)$ is equivalent to $(\beta-a)+\frac{a+1}{p'}>0.$ As $\omega<0,$ we obtain $(\beta-a)\omega+\frac{a+1}{p'}\omega <0,$ which is the same as $$-\frac{\beta -a}{p'}(a+1)-\frac{\beta-a}{q}(b+1)+\frac{a+1}{p'}\omega<0$$ or 
\begin{eqnarray}
\label{e1}
\frac{a+1}{p'}\omega - \frac{\beta-a}{p'}(a+1)<\frac{\beta-a}{q}(b+1).
\end{eqnarray}
We also have that $-\alpha q<b+1$ or $-\alpha<\frac{b+1}{q}$ is equivalent to $-\alpha\omega-\frac{b+1}{q}\omega>0$ or \\$\alpha\frac{a+1}{p'}+\alpha\frac{b+1}{q}-\frac{b+1}{q}\omega>0$ which is the same as 
\begin{eqnarray}
\label{e2}
\frac{b+1}{q}\omega-\alpha\frac{b+1}{q}<\alpha\frac{a+1}{p'}.
\end{eqnarray}
From \eqref{e1} and \eqref{e2}, we see that we can find two numbers $r$ and $s$ such that
\begin{eqnarray}
\label{e3}
\frac{a+1}{p'}\omega-\frac{\beta-a}{p'}(a+1)<\omega s+(\beta-a)(r-s)<\frac{\beta-a}{q}(b+1)\quad\mbox{and}
\end{eqnarray}
\begin{eqnarray}
\label{e4}
\frac{b+1}{q}\omega-\alpha\frac{b+1}{q}<\omega r+\alpha(s-r)<\alpha\frac{a+1}{p'}.
\end{eqnarray}
\eqref{e3} is equivalent to the inequality
\begin{eqnarray}
\label{e5}
-\frac{\beta-a}{\omega}\left[-\frac{b+1}{q}+r-s\right]<s<\frac{a+1}{p'}+\frac{\beta-a}{\omega}\left[-\frac{a+1}{p'}+s-r\right]
\end{eqnarray}
while \eqref{e4} is equivalent to
\begin{eqnarray}
\label{e6}
\frac{\alpha}{\omega}\left[\frac{a+1}{p'}+r-s\right]<r<\frac{b+1}{q}+\frac{\alpha}{\omega}\left[-\frac{b+1}{q}+r-s\right].
\end{eqnarray}
Let $$t=\frac{-\frac{a+1}{p'}+s-r}{\omega}.$$ It is easy to see that $$1-t=\frac{r-s-\frac{b+1}{q}}{\omega}.$$ Hence \eqref{e5} becomes
\begin{eqnarray}
\label{e7}
-(\beta-a)(1-t)<s<\frac{a+1}{p'}+(\beta-a)t
\end{eqnarray}
and \eqref{e6} becomes
\begin{eqnarray}
\label{e8}
-\alpha t<r<\frac{b+1}{q}+\alpha(1-t).
\end{eqnarray}
As $\gamma>0,$ we can even choose $r$ and $s$ in \eqref{e7} and \eqref{e8} so that $0<r-s<\frac{b+1}{q}.$ Note that this choice clearly gives us that $0<t<1.$

Next, we observe that the operator $H_{\alpha,\beta,\gamma}$ can be represented as $$H_{\alpha,\beta,\gamma}f(x) = \int_0^{\infty}\!\!\!K(x,y)f(y)y^{a}\mathrm{d}y\quad\mbox{where}\quad K(x,y)=\frac{y^{\beta-a}x^{\alpha}}{(x+y)^{\gamma}}.$$ Let us define  $h_1(x) = x^{-s}$ and $h_2(y) = y^{-r}.$ Applying Okikiolu's test to $H_{\alpha,\beta,\gamma}$ we obtain
\begin{eqnarray*}
\int_0^{\infty}\!\!\!K(x,y)^{tp'}y^{-sp'}y^a\mathrm{d}y& = &\int_0^{\infty}\!\frac{x^{\alpha tp'}y^{(\beta-a)tp'}y^{-sp'+a}}{(x+y)^{t\gamma p'}}\mathrm{d}y\\
&=& x^{\alpha tp'-t\gamma p'-sp'+(\beta-a)tp'+a+1}\int_0^{\infty}\!\frac{y^{-sp'+(\beta-a)tp'+a}}{(1+y)^{t\gamma p'}}\mathrm{d}y.
\end{eqnarray*}
We observe that the right inequality in \eqref{e7} provides $a+1+(\beta-a)tp'-sp'>0$. From the definition of $\omega, t$ and the first inequality in \eqref{e8}, we have that
\begin{eqnarray*}
t\gamma p' + sp' - (\beta-a)tp'-a-1 & = & (\gamma-\beta+a)tp'+sp'-a-1\\
& = & \left(\alpha+\frac{a+1}{p'}+\frac{b+1}{q}\right)tp'+sp'-a-1\\
& = & (\alpha - \omega)tp'+sp'-a-1\\
& = & \alpha tp'-\omega tp'+sp'-a-1\\
& = & \alpha tp'+\left(\frac{a+1}{p'}+r-s\right)p'+sp'-a-1\\
& = & \alpha tp'+rp'>0.
\end{eqnarray*}
It follows that 
\begin{eqnarray*}
\int_0^{\infty}\!\!\!K(x,y)^{tp'}y^{-sp'}y^a\mathrm{d}y & = & B(-sp'+(\beta-a)tp'+a+1, \alpha tp'+rp') x^{-rp'}\\
& = & B(sp'+(\beta-a)tp'+a+1,\alpha tp'+rp')h_2^{p'}(x).
\end{eqnarray*}
In the same way, we obtain
\begin{eqnarray*}
\int_0^{\infty}\!\!\![K(x,y)]^{(1-t)q}h_2^q(x)x^b\mathrm{d}x & = & \int_0^{\infty}\!\frac{x^{(1-t)\alpha q}y^{(\beta-a)(1-t)q}x^{-rq}x^b}{(x+y)^{\gamma(1-t)q}}\mathrm{d}x\\
& = & y^{(\beta-a)(1-t)q-\gamma(1-t)q-rq+\alpha(1-t)q+b+1}\int_0^{\infty}\!\frac{x^{-rq+\alpha(1-t)q+b}}{(1+x)^{\gamma(1-t)q}}\mathrm{d}x.
\end{eqnarray*}
From the second inequality in \eqref{e8}, we get $-rq+\alpha(1-t)q+b+1>0.$ From the definition of $\omega$ and $1-t,$ and the first inequality in \eqref{e7}, we obtain
\begin{eqnarray*}
\gamma(1-t)q+rq-\alpha(1-t)q-b-1 & = & (\gamma-\alpha)(1-t)q+rq-b-1\\
& = & \left(\beta+1 - \frac{a+1}{p}+\frac{b+1}{q}\right)(1-t)q+rq-b-1\\
& =& \left(\beta -a + \frac{a+1}{p'}+\frac{b+1}{q}\right)(1-t)q+rq-b-1\\
& = & (\beta-a-\omega)(1-t)q+rq-b-1\\
& = & (\beta-a)(1-t)q-\omega(1-t)q+rq-b-1\\
& = & (\beta-a)(1-t)q+\left(\frac{b+1}{q}-r+s\right)q+rq-b-1\\
& = & (\beta-a)(1-t)q+sq>0.
\end{eqnarray*}
Hence
\begin{eqnarray*}
\int_0^{\infty}\!\!\![K(x,y)]^{(1-t)q}h_2^q(x)x^b\mathrm{d}x & = & B(-rq+\alpha(1-t)q+b+1, (\beta-a)(1-t)q+sq)y^{-sq}\\
& = & B(-rq+\alpha(1-t)q+b+1, (\beta-a)(1-t)q+sq)h_1^q(y)
\end{eqnarray*}
and the proof is complete.
\end{proof}
We next consider the limit case $p=1$.
\begin{lemma}\label{lem:sufflimitcase}
Let $1\le q<\infty,$ $\alpha,\beta,\gamma \in\mathbb{R};$ and $a,b>-1.$ Assume that $$\gamma = \alpha+\beta-a + \frac{b+1}{q}$$ and that $$\beta+1-\gamma<a+1<\beta + 1.$$ Then the operator $H_{\alpha, \beta,\gamma}$ is bounded from $ L^1((0,\infty), y^a\mathrm{d}y)$ to $ L^q((0,\infty), y^b\mathrm{d}y).$
\end{lemma}
\begin{proof}
Assume that $\gamma=\alpha+\beta-a+\frac{b+1}{q}$ and $\beta+1-\gamma<a+1<\beta+1$ then $$\beta+1-\gamma<a+1<\beta+1$$ is equivalent to $$-\alpha q<b+1<q(\gamma-\alpha).$$ In this case $\omega = \alpha+\beta-\gamma-a=-\frac{b+1}{q}<0$ and $t=\frac{s-r}{\omega}$. The inequalities (\ref{e7}) and (\ref{e8}) reduce to

\begin{eqnarray}
\label{e7limit}
-(\beta-a)(1-t)<s<(\beta-a)t
\end{eqnarray}
and 
\begin{eqnarray}
\label{e8limit}
-\alpha t<r<\frac{b+1}{q}+\alpha(1-t)
\end{eqnarray}
respectively.

We first check the first condition in Lemma \ref{lem:okikiolulimitcase}, that is there exists a constant $C_1>0$ such that 
$$\sup_{0<y<\infty}\,h_1(y)K(x,y)^t\le C_1h_2(x)\,\,\,\textrm{for almost every}\,\,\,x\in \mathbb{R}$$ or equivalently,$$ \sup_{0<y<\infty}y^{-s}\left(\frac{y^{\beta-a}x^{\alpha}}{(x+y)^{\gamma}}\right)^tx^{r}\le C_1$$
for almost every $x\in \mathbb{R}$. This is the case since the power in the denominator is equal to the sum of the exponents in the numerator. Indeed, we have 
\Beas\gamma t &=& (\alpha+\beta-a-\omega)t=\alpha t+(\beta-a)t-\omega t\\ &=& \alpha t+(\beta-a)t+r-s.
\Eeas
That the second condition in Lemma \ref{lem:okikiolulimitcase} is satisfied follows as in the proof of the previous lemma with the help of inequalities (\ref{e7limit}) and (\ref{e8limit}).
\end{proof}
\subsection{Proof of Theorem \ref{thm:main1} and the endpoint cases}

The proof of Theorem \ref{thm:main1} follows easily from Lemma \ref{lem:necessadjoint} and Lemma \ref{lem:suffnotlimit} in the case $p>1$. For $1=p\le q<\infty$, the proof follows from Lemma \ref{lem:necessdirect}, Lemma \ref{lem:sufflimitcase} and the observation made in Lemma \ref{lem:equivcond}.

We prove here Theorem \ref{thm:main2pinfty} and Theorem \ref{thm:main3inftyinfty}.
\begin{proof}[Proof of Theorem \ref{thm:main2pinfty}]
Let us start by the sufficiency. Assume that the parameters satisfy the conditions of the theorem. Let $f\in L_a^p$. Then using the H\"older's inequality and Lemma \ref{lem:betafunctionconditions}, we obtain
\Beas
\left|H_{\alpha,\beta,\gamma}f(x)\right| &=& \left|x^\alpha\int_0^\infty\frac{f(y)}{(x+y)^\gamma}y^\beta \mathrm{d}y\right|\\ &\le& \|f\|_{p,a}x^\alpha \left(\int_0^\infty\frac{y^{p'(\beta-a)}}{(x+y)^{p'\gamma}}y^a\mathrm{d}y\right)^{\frac{1}{p'}}\\ &=& [B\left(p'(\beta-a)+a+1, p'\alpha\right)]^{\frac{1}{p'}}]\|f\|_{p,a}.
\Eeas
Hence $$\|H_{\alpha,\beta,\gamma}f\|_{L^\infty}\le [B\left(p'(\beta-a)+a+1, p'\alpha\right)]^{\frac{1}{p'}}]\|f\|_{p,a}.$$
Conversely, assume that the operator $H_{\alpha,\beta,\gamma}$ is bounded from $L_a^p$ to $L^\infty$. Then its adjoint $H_{\alpha,\beta,\gamma}^*$ is bounded from $L^1$ to $L_a^{p'}$. One easily check that in this case, $$H_{\alpha,\beta,\gamma}^*f(y)=y^{\beta-a}\int_0^\infty \frac{f(x)}{(x+y)^\gamma}x^{\alpha} \mathrm{d}x.$$

Let $R>0$ and define $f_R(x):=f(Rx).$ Then it is easy to see that if $f\in  L^1((0,\infty))$ then $f_R\in  L^1((0,\infty))$ and $\|f_R\|_{L^1}=R^{-1}\|f\|_{L^1}.$
As in the proof of Lemma \ref{lem:necessdirect} we obtain
$$
H_{\alpha,\beta,\gamma}^*f_R(x)= R^{\gamma-\beta-\alpha+a-1}H_{\alpha,\beta,\gamma}^*f(Rx).
$$
So
$$
\|H_{\alpha,\beta,\gamma}^*f_R\|_{p',a}^{p'} = R^{p'(\gamma-\alpha-\beta+a-1)-a-1}\|H_{\alpha,\beta,\gamma}^*f\|_{p',a}^{p'}.
$$
Now, that for any $f\in L^1((0,\infty))$, $\|H_{\alpha,\beta,\gamma}^*f_R\|_{p',a}\le C\|f_R\|_{L^1}$ is equivalent to $$R^{\gamma-\alpha-\beta+a-1-\frac{a+1}{p'}}\|H_{\alpha,\beta,\gamma}f\|_{p',a}\le CR^{-1}\|f\|_{L^1}$$ which is the same as $$R^{\gamma-\alpha-\beta+a-1-\frac{a+1}{p'}+1}\|H_{\alpha,\beta,\gamma}^*f\|_{p',a}\le C\|f\|_{L^1}$$ for any $f\in L^1((0,\infty))$. This is only possible if $\gamma-\alpha-\beta+a-\frac{a+1}{p'}=0$. That is $\gamma=\alpha+\beta+1-\frac{a+1}{p}$.

Let us take again $f(y)=\chi_{[1,2]}(y)$ as test function and proceed as in the proof of Lemma \ref{lem:necessadjoint}. We obtain that $-p'(\beta-a)<a+1<p'(\gamma-\beta+a)$. Combining the equality $\gamma=\alpha+\beta+1-\frac{a+1}{p}$ and the inequality $0<a+1<p'(\gamma-\beta+a)$, we obtain $0<a+1<p'\alpha$ and hence that $\alpha>0$. Note that under the equality $\gamma=\alpha+\beta+1-\frac{a+1}{p}$, the inequality  $-p'(\beta-a)<a+1<p'(\gamma-\beta+a)$ is equivalent to $-p\alpha<a+1<p(\beta+1)$. The proof is complete.
\end{proof}
\begin{proof}[Proof of Theorem \ref{thm:main3inftyinfty}]
Following Lemma \ref{lem:okikiolulimitcaseinfty}, we only have to find necessary and suficient conditions on the parameters so that 
$$\sup_{0<x<\infty}\int_0^\infty\, \frac{x^\alpha y^\beta}{(x+y)^\gamma}\mathrm{d}y<\infty.$$
Following Lemma \ref{lem:betafunctionconditions}, this is the case if and only if $\beta+1>0$, $\gamma-\beta-1>0$ and $\alpha+\beta+1-\gamma=0$. This is equivalent to $\gamma=\alpha+\beta+1$, $\alpha>0$ and $\beta>-1$. Moreover, follwing Lemma \ref{lem:okikiolulimitcaseinfty}, we have 
\Beas
\|H_{\alpha,\beta,\gamma}\|_{L^\infty\rightarrow L^\infty} &=& \sup_{0<x<\infty}\int_0^\infty\, \frac{x^\alpha y^\beta}{(x+y)^\gamma}\mathrm{d}y\\ &=& B(\beta+1,\alpha).
\Eeas
The proof is complete.
\end{proof}
\subsection{Sharp Norm for Generalized Hilbert Operator}
Let us start by proving the following estimate.
\begin{lemma}\label{lem:stepsharp}
Let $1<p<\infty$, and let $\alpha,\beta$ and $\gamma$ be real numbers such that $\gamma=\alpha+\beta+1>0$. Let $a>-1$ and let $0<\varepsilon<p(\beta+1)-(a+1)$. Then
$$\int_1^{\infty}\!\!\!x^{a+\alpha-\frac{a+1+\xi}{p}}\left(\int_0^1\!\!\frac{y^{\beta-\frac{a+1+\xi}{p}}}{(x+y)^{\gamma}}\mathrm{d}y\right)\mathrm{d}x\le \frac{1}{\beta+\frac{a+1+\xi}{p'}-a}\times \frac{1}{\beta+1-\frac{a+1+\xi}{p}}.$$
\end{lemma}
\begin{proof}
We easily obtain
\Beas
I &:=& \int_1^{\infty}\!\!\!x^{a+\alpha-\frac{a+1+\xi}{p}}\left(\int_0^1\!\!\frac{y^{\beta-\frac{a+1+\xi}{p}}}{(x+y)^{\gamma}}\mathrm{d}y\right)\mathrm{d}x\\ &\le& \int_1^{\infty}\!\!\!x^{a+\alpha-\gamma-\frac{a+1+\xi}{p'}}\left(\int_0^1\!\!y^{\beta-\frac{a+1+\xi}{p}}\mathrm{d}y\right)\mathrm{d}x\\ &=& \int_1^{\infty}\!\!\!x^{a-\beta-1-\frac{a+1+\xi}{p'}}\mathrm{d}x\int_0^1\!\!y^{\beta-\frac{a+1+\xi}{p}}\mathrm{d}y\\ &=& \frac{1}{\beta+\frac{a+1+\xi}{p'}-a}\times \frac{1}{\beta+1-\frac{a+1+\xi}{p}}.
\Eeas
\end{proof}
We next prove the following.
\begin{theorem}
Let $1<p<\infty$ and $a>-1.$ Assume that $-p\alpha<a+1<p(\beta+1)$ and $\gamma = \alpha+\beta+1.$ Then $$\|H_{\alpha,\beta,\gamma}\|_{L^p_a\rightarrow L^p_a} = B\left(\beta+1-\frac{a+1}{p},\alpha+\frac{a+1}{p}\right).$$
\end{theorem}
\begin{proof}
Note that $\|H_{\alpha,\beta,\gamma}\|_{L^p_a\rightarrow L^p_a} \le B\left(\beta+1-\frac{a+1}{p},\alpha+\frac{a+1}{p}\right)$ follows using Schur's test and it is in proof of the sufficient part of Theorem \ref{thm:main1}.  To prove that $B\left(\beta+1-\frac{a+1}{p},\alpha+\frac{a+1}{p}\right)$ is sharp, we proceed by contradiction. Assume that\\ $B\left(\beta+1-\frac{a+1}{p},\alpha+\frac{a+1}{p}\right)$ is not sharp. i.e. there exists $0<K< B\left(\beta+1-\frac{a+1}{p},\alpha+\frac{a+1}{p}\right)$ such that for any $f\in L^p_a((0,\infty)),$ $$\|H_{\alpha,\beta,\gamma}f\|_{L^p_a}\le K\|f\|_{L^p_a}$$ or equivalently, for any $f\in L^p_a((0,\infty))$ and $g\in L_a^{p'}((0,\infty)),$
\begin{eqnarray}
\label{s1}
K\|f\|_{L^p_a}\|g\|_{L^{p'}_a}\ge \int_0^{\infty}\!\!\!g(x)H_{\alpha,\beta,\gamma}f(x)x^a\mathrm{d}x
\end{eqnarray}
Let $0<\varepsilon<p(\beta+1)-(a+1)$ and define\\ $\begin{array}{lcr}f(x)=\left\{\begin{array}{lcr}0 & \mbox{if} & 0<x<1\\ x^{-\frac{a+1+\xi}{p}} & \mbox{if} & x\ge 1\end{array}\right.&\quad\mbox{and} &\quad g(x)= \left\{\begin{array}{lcr}0 & \mbox{if} & 0<x<1\\ x^{-\frac{a+1+\xi}{p'}} & \mbox{if} & x\ge 1\end{array}\right.\end{array}.$\\ Then $$\|f\|_{L^p_a} = \frac{1}{\xi^{\frac{1}{p}}}\qquad\mbox{and}\qquad\|g\|_{L^{p'}_a} = \frac{1}{\xi^{\frac{1}{p'}}}.$$ Substituting these into \eqref{s1} and using Lemma \ref{lem:stepsharp}, we obtain 
\begin{eqnarray*}
\frac{K}{\xi} & \ge & \int_1^{\infty}\!\!\!x^{a+\alpha-\frac{a+1+\xi}{p'}}\left(\int_1^{\infty}\!\!\frac{y^{\beta-\frac{a+1+\xi}{p}}}{(x+y)^{\gamma}}\mathrm{d}y\right)\mathrm{d}x\\
& = &\int_1^{\infty}\!\!\!x^{a+\alpha-\frac{a+1+\xi}{p'}}\left(\int_0^{\infty}\!\!\frac{y^{\beta-\frac{a+1+\xi}{p}}}{(x+y)^{\gamma}}\mathrm{d}y\right)\mathrm{d}x - \int_1^{\infty}\!\!\!x^{a+\alpha-\frac{a+1+\xi}{p}}\left(\int_0^1\!\!\frac{y^{\beta-\frac{a+1+\xi}{p}}}{(x+y)^{\gamma}}\mathrm{d}y\right)\mathrm{d}x\\
& = & B\left(\beta+1-\frac{a+1+\xi}{p},\gamma-\beta-1+\frac{a+1+\xi}{p}\right)\int_1^{\infty}\!\!\!x^{a+\alpha-\frac{a+1+\xi}{p'}-\gamma+\beta+1-\frac{a+1+\xi}{p}}\mathrm{d}x\\
& &\quad - \int_1^{\infty}\!\!\!x^{a+\alpha-\frac{a+1+\xi}{p'}}\left(\int_0^1\!\!\frac{y^{\beta-\frac{a+1+\xi}{p}}}{(x+y)^{\gamma}}\mathrm{d}y\right)\mathrm{d}x\\
& \ge & \frac{1}{\xi}B\left(\beta+1-\frac{a+1+\xi}{p},\alpha+\frac{a+1+\xi}{p}\right) + \frac{1}{a-\beta-\frac{a+1+\xi}{p'}}\times \frac{1}{\beta+1-\frac{a+1+\xi}{p}}
\end{eqnarray*}
So $$K\ge B\left(\beta+1-\frac{a+1+\xi}{p},\alpha+\frac{a+1+\xi}{p}\right)+\frac{\xi}{\left(a-\beta-\frac{a+1+\xi}{p'}\right)\left(\beta+1-\frac{a+1+\xi}{p}\right)}.$$ Thus letting $\xi\longrightarrow 0,$ we obtain $K\ge B\left(\beta+1-\frac{a+1}{p},\alpha+\frac{a+1}{p}\right).$ Hence a contradiction.
\end{proof}
In the limit case $p=1$, we prove the following.
\begin{theorem}
Let $a>-1.$ Assume that $-\alpha<a+1<\beta+1$ and $\gamma = \alpha+\beta+1.$ Then $$\|H_{\alpha,\beta,\gamma}\|_{L^1_a\rightarrow L^1_a} = B\left(\beta-a,\alpha+a+1\right).$$
\end{theorem}
\begin{proof}
Recall that the kernel of $H_{\alpha,\beta,\gamma}$ with respect to the measure $y^a\mathrm{d}y$ is $K(x,y)=\frac{x^\alpha y^{\beta-a}}{(x+y)^\gamma}$. It follows from Lemma \ref{lem:okikiolulimitcasesharp} that 
\Beas
\|H_{\alpha,\beta,\gamma}\|_{L^1_a\rightarrow L^1_a} &=& \sup_{0<y<\infty}\int_0^\infty \frac{x^\alpha y^{\beta-a}}{(x+y)^\gamma}x^a \mathrm{d}x\\ &=& B\left(\beta-a,\alpha+a+1\right).
\Eeas
\end{proof}
\section{Boundedness of a family of positive Bergman-type Operators}
We recall that the integral operator $T^+$ is given by $$T^+f(x+iy) = y^{\alpha}\int_{\mathbb{R}_+^2}\!\!\frac{f(w)}{|z-\bar{w}|^{1+\gamma}}v^{\beta}\mathrm{d}u\mathrm{d}v\quad\mbox{where} \quad w=u+iv, z=x+iy.$$ 
\subsection{Sufficiency for Boundedness of $T^+_{\alpha,\beta,\gamma}$}
Let us start by the following lemma.
\begin{lemma}
\label{b1}
Suppose that $a,b>-1,1\le p<\infty$ and $1\le q\le r \le\infty.$ Then the  operator $T^+$ is bounded from $L_a^{p,q}(\mathbb{R}_+^2)$ to $L_b^{p,r}(\mathbb{R}_+^2)$ if the operator $H_{\alpha,\beta,\gamma}$ is bounded from $L_a^{q}((0,\infty))$ to $L_b^{r}((0,\infty))$.
\end{lemma}
\begin{proof}
For simplicity, we shall write $f(x+iy):= f_y(x).$ Then $$T^+f(x+iy) = \left(T^+f\right)_y(x) = y^{\alpha}\int_0^{\infty}\left(\int_{\mathbb{R}}\frac{f_v(u)}{|(x-u)+i (y+v)|^{1+\gamma}}\mathrm{d}u\right)v^{\beta}\mathrm{d}v.$$
The idea is to prove that for any $f\in L_a^{p,q}(\mathbb{R}_+^2)$, and any $0<y<\infty$, $$\left\|\left(T^+f\right)_y\right\|_{L^p(\mathrm{d}x)}\le C_{\gamma}H_{\alpha,\beta,\gamma}(\|f_v\|_{L^p})(y).$$

Let $K(z)=\frac{1}{z}.$ Then
\begin{eqnarray*}
\left(T^+f\right)_y(x) & = & y^{\alpha}\int_0^{\infty}\left(\int_{\mathbb{R}}|K[(x-u)+i (y+v)]|^{1+\gamma}f_v(u)\mathrm{d}u\right)v^{\beta}\mathrm{d}v\\
& = & y^{\alpha}\int_0^{\infty}\left(\int_{\mathbb{R}}|K_{y+v}(x-u)|^{1+\gamma}f_v(u)\mathrm{d}u\right)v^{\beta}\mathrm{d}v\\
& = & y^{\alpha}\int_0^{\infty}\left(|K_{y+v}|^{1+\gamma}*f_v\right)(x)v^{\beta}\mathrm{d}v.
\end{eqnarray*}
Now, using Minkowski's inequality, Young's inequality and Lemma \ref{lem:integkernel}, we obtain
\begin{eqnarray*}
\left\|\left(T^+f\right)_y\right\|_{L^p(\mathrm{d}x)} & = & \left(\int_{\mathbb{R}}\left|y^{\alpha}\int_0^{\infty}\left(|K_{y+v}|^{1+\gamma}*f_v\right)(x)v^{\beta}\mathrm{d}v\right|^p\mathrm{d}x\right)^{\frac{1}{p}}\\
&\le& \left(\int_{\mathbb{R}}\left(y^{\alpha}\int_0^{\infty}\left|\left(|K_{y+v}|^{1+\gamma}*f_v\right)(x)\right|v^{\beta}\mathrm{d}v\right)^p\mathrm{d}x\right)^{\frac{1}{p}}\\
&=& y^{\alpha}\left(\int_{\mathbb{R}}\left(\int_0^{\infty}\left|\left(|K_{y+v}|^{1+\gamma}*f_v\right)(x)\right|v^{\beta}\mathrm{d}v\right)^p\mathrm{d}x\right)^{\frac{1}{p}}\\
&\le& y^{\alpha}\int_0^{\infty}\left(\int_{\mathbb{R}}\left|\left(|K_{y+v}|^{1+\gamma}*f_v\right)(x)\right|^p\mathrm{d}x\right)^{\frac{1}{p}}v^{\beta}\mathrm{d}v\\
&=& y^{\alpha}\int_0^{\infty}\left\|\,|K_{y+v}|^{1+\gamma}*f_v\right\|_{L^p}v^{\beta}\mathrm{d}v\\
&\le& y^{\alpha}\int_0^{\infty}\left\||K_{y+v}|^{1+\gamma}\|_{L^1}\|f_v\right\|_{L^p}v^{\beta}\mathrm{d}v\\
&=&C_{\gamma} y^{\alpha}\int_0^{\infty}(y+v)^{-\gamma}\left\|f_v\right\|_{L^p}v^{\beta}\mathrm{d}v\\
&=&C_{\gamma} y^{\alpha}\int_0^{\infty}\frac{\left\|f_v\right\|_{L^p}}{(y+v)^{\gamma}}v^{\beta}\mathrm{d}v\\
&=&C_{\gamma}H_{\alpha,\beta,\gamma}(\|f_v\|_{L^p})(y).
\end{eqnarray*}
That is $\left\|\left(T^+f\right)_y\right\|_{L^p(\mathrm{d}x)}\le C_{\gamma}H_{\alpha,\beta,\gamma}(\|f_v\|_{L^p})(y)$ as we wanted. Thus the boundedness of $H_{\alpha,\beta,\gamma}$ from $L_a^{q}((0,\infty))$ to $L_b^{r}((0,\infty))$ implies the boundedness of  $T^+$ from $L_a^{p,q}(\mathbb{R}_+^2)$ to $L_b^{p,r}(\mathbb{R}_+^2)$. The proof is complete.
\end{proof}
It follows from the above lemma, Theorem \ref{thm:main1} and Theorem \ref{thm:main2pinfty} that the following hold.
\begin{lemma}
\label{b11}
Suppose that $a,b>-1,1\le p<\infty$ and $1\le q\le r \le\infty.$ Assume that $\alpha,\beta$ and $\gamma$ are real numbers satisfying the conditions in Theorem \ref{thm:main2}, Theorem \ref{thm:main3}, Theorem \ref{thm:main4}, Theorem \ref{thm:main5}, Theorem \ref{thm:main6} or Theorem \ref{thm:main7}. Then the  operator $T^+$ is bounded from $L_a^{p,q}(\mathbb{R}_+^2)$ to $L_b^{p,r}(\mathbb{R}_+^2)$.
\end{lemma}
\subsection{Necessity for Boundedness of $T^+_{\alpha,\beta,\gamma}$} We start by the following lemma.
\begin{lemma}
\label{b2}
Let $1<p<\infty$, $1<q\le r<\infty$ and $a,b>-1$. Assume that $T^+_{\alpha,\beta,\gamma}$ is bounded from $L^{p,q}_a(\mathbb{R}^2_+)$ to $L^{p,r}_b(\mathbb{R}^2_+).$ Then the parameters satisfy $$\gamma = \alpha+\beta+1-\frac{a+1}{q}+\frac{b+1}{r}\quad\mbox{and}\quad -q(\gamma-\beta-1)<a+1<q(\beta+1).$$ 
\end{lemma}
\begin{proof}
Let $R>0$ and set $f_R(z):=f(Rz).$ Then it is easily seen that if $f\in L^{p,q}_a(\mathbb{R}_+^2)$, then $f_R\in L^{p,q}_a(\mathbb{R}_+^2)$ and $$\|f_R\|_{L^{p,q}_a}=R^{-\frac{a+1}{q}-\frac{1}{p}}\|f\|_{l^{p,q}_a}.$$ One easily checks that $$\left(T^+f_R\right)(z) = R^{\gamma-\beta-\alpha-1}T^+f(Rz)$$ and thus $$\|T^+f_R\|_{L^{p,r}_b} = R^{\gamma-\beta-\alpha-1-\frac{1}{p}-\frac{b+1}{r}}\|T^+f\|_{L^{p,r}_b}.$$ Now that $T^+$ is bounded implies that there exists a constant $C>0$ such that for any $f\in L^{p,r}_a(\mathbb{R}_+^2)$, $$\|T^+f_R\|_{L^{p,r}_b}\le C\|f_R\|_{L_a^{p,r}}.$$ Therefore $$R^{\gamma-\beta-\alpha-1-\frac{1}{p}-\frac{b+1}{r}}\|T^+f\|_{L^{p,r}_b}\le CR^{-\frac{a+1}{q}-\frac{1}{p}}\|f\|_{L^{p,q}_a}$$ or equivalently, $$R^{\gamma-\beta-\alpha-1-\frac{b+1}{r}+\frac{a+1}{q}}\|T^+f\|_{L^{p,r}_b}\le C\|f\|_{L^{p,q}_a}.$$ As this holds for any $f\in L^{p,q}_a(\mathbb{R}_+^2)$ and any $R>0$, we necessarily have  that $$\gamma-\beta-\alpha-1-\frac{b+1}{r}+\frac{a+1}{q}=0.$$ That is $\gamma=\beta+\alpha+1+\frac{b+1}{r}-\frac{a+1}{q}$.

To check the other condition, let us set $$f_y(x) = \chi_{[-\frac{1}{4},\frac{1}{4}]}(x)\chi_{[1,2]}(y).$$ For $-\frac{1}{4}\le x\le \frac{1}{4}$ and $0<y\le 1,$ it is easy to check that $$\int_{-\frac{1}{4}}^{\frac{1}{4}}\!\frac{\mathrm{d}u}{[(x-u)^2+(y+v)^2]^{\frac{1+\gamma}{2}}}\gtrsim (y+v)^{-\gamma}.$$ It follows that for $-\frac{1}{4}\le x\le \frac{1}{4}$ and $0<y\le 1,$ $$(T^+f)_y(x) = y^{\alpha}\int_1^2\int_{-\frac{1}{4}}^{\frac{1}{4}}\!\frac{v^{\beta}\mathrm{d}u\mathrm{d}v}{[(x-u)^2+(y+v)^2]^{\frac{1+\gamma}{2}}}\gtrsim\frac{y^{\alpha}}{(1+y)^{\gamma}}.$$ Hence, as $T^+$ is bounded from $L^{p,q}_a(\mathbb{R}_+^2)$ to $L^{p,r}_b(\mathbb{R}_+^2),$ $$\int_0^1\left(\int_{-\frac{1}{4}}^{\frac{1}{4}}\!\!|T^+f(x+iy)|^p \mathrm{d}x\right)^{\frac{r}{p}}y^b\mathrm{d}y\le\|T^+f\|^r_{L^{p,r}_b}\le C\|f\|^r_{L^{p,q}_a}<\infty.$$ This implies that $$\int_0^1\frac{y^{r\alpha}}{(1+y)^{r\gamma}}y^b\mathrm{d}y< \infty$$ and as $$\int_0^1y^{r\alpha + b}\mathrm{d}y\lesssim\int_0^1\frac{y^{r\alpha}}{(1+y)^{\gamma}}y^b\mathrm{d}y,$$ we should have $$\int_0^1y^{r\alpha + b}\mathrm{d}y<\infty$$ and this is possible only if $r\alpha+b+1>0,$ that is $-r\alpha<b+1.$ Using that $\gamma=\beta+\alpha+1+\frac{b+1}{r}-\frac{a+1}{q}$, we obtain that the latter is equivalent to $-q(\gamma-\beta-1)<a+1$. This gives us the left inequality in the second condition. To prove the right inequality, we observe that the boundedness of $T^+$ from $L^{p,q}_a(\mathbb{R}_+^2)$ to $L^{p,r}_b(\mathbb{R}_+^2)$ is equivalent to the boundedness of its adjoint $(T^+)^*$ from $L^{p',r'}_b(\mathbb{R}_+^2)$ to $L^{p',q'}_a(\mathbb{R}_+^2)$. One easily checks that $$(T^+)^*g(u+iv) = v^{\beta-a}\int_{\mathbb{R}^2_+}\frac{g(x+iy)}{|(x-u)+i(y+v)|^{1+\gamma}}y^{\alpha +b}\mathrm{d}y.$$ As above, we take $$g_u(v) = \chi_{[-\frac{1}{4},\frac{1}{4}]}(u)\chi_{[1,2]}(v)$$ and obtain for $-\frac{1}{4}\le u\le \frac{1}{4}$ and $0<v\le 1,$ $$(T^+)^*g_v(u)\gtrsim\frac{v^{\beta-a}}{(1+v)^{\gamma}}$$ and so following the lines of the proof for the necessity of the left inequality, we are led to $$\int_0^1\!\!\!v^{q'(\beta-a)+a}\mathrm{d}v<\infty$$ and this holds only if $q'(\beta-a)+a+1>0$ or equivalently, $a+1<q(\beta+1).$ 
%Using that $\gamma =\alpha+\beta+1+\frac{b+1}{q}-\frac{a+1}{p},$ we obtain that the latter inequality is the same as $b+1<r(\gamma-\alpha)$ and the proof is complete. 
\end{proof}
We next prove the following.
\begin{lemma}
\label{b3}
Let $1<p,q<\infty$ and $a>-1$. Assume that $T^+_{\alpha,\beta,\gamma}$ is bounded from $L^{p,q}_a(\mathbb{R}^2_+)$ to $L^{p,\infty}(\mathbb{R}^2_+).$ Then the parameters satisfy $$\gamma = \alpha+\beta+1-\frac{a+1}{q}\quad\mbox{and}\quad -q(\gamma-\beta-1)<a+1<q(\beta+1).$$ 
\end{lemma}
\begin{proof}
We write again $f_R(z)=f(Rz)$. Following the lines of the proof of the previous lemma, we obtain $$\|T^+f_R\|_{L^{p,\infty}} = R^{\gamma-\beta-\alpha-1-\frac{1}{p}-\frac{b+1}{r}}\|T^+f\|_{L^{p,\infty}}.$$ Now that $T^+$ is bounded implies that for some constant $C>0$, $$\|T^+f_R\|_{L^{p,\infty}}\le C\|f_R\|_{L^{p,r}_a}\quad\mbox{for all $f\in L^{p,r}_a(\mathbb{R}_+^2)$.}$$ Therefore $$R^{\gamma-\beta-\alpha-1-\frac{1}{p}}\|T^+f\|_{L^{p,\infty}}\le CR^{-\frac{a+1}{q}-\frac{1}{p}}\|f\|_{L^{p,q}_a}$$ or equivalently, $$R^{\gamma-\beta-\alpha-1+\frac{a+1}{q}}\|T^+f\|_{L^{p,\infty}}\le C\|f\|_{L^{p,q}_a}.$$ As this holds for any $f\in L^{p,q}_a(\mathbb{R}_+^2)$ and any $R>0$, we necessarily have  that $\gamma-\beta-\alpha-1+\frac{a+1}{q}=0$. That is $\gamma=\beta+\alpha+1-\frac{a+1}{q}$.

To obtain the other condition, we proceed as in the proof of the previous lemma. We take again $f(x+iy)=\chi_{[-\frac{1}{4},\frac{1}{4}]}(x)\chi_{[1,2]}$ and obtain from the boundedness of $T^+$ that there is a constant $C>0$ such that
$$\sup_{0<y<1}\left(\int_{-1/4}^{1/4}\left|T^+f(x+iy)\right|^p\mathrm{d}x\right)^{1/p}\le \|T^+f\|_{L^{p,\infty}}\le C\|f\|_{L_a^{p,q}}<\infty.$$
Hence $$\sup_{0<y<1}y^\alpha \lesssim \sup_{0<y<1}\frac{y^\alpha}{(1+y)^\gamma}<\infty$$
and this clearly implies that $\alpha>0$. From the equality $\gamma=\beta+\alpha+1-\frac{a+1}{q}$, it is easy to see that $\alpha>0$ is the same as $-q(\gamma-\beta-1)<a+1$. This give the left inequality in the second condition. To prove the right inequality, we observe that the boundedness of $T^+$ from $L^{p,q}_a(\mathbb{R}_+^2)$ to $L^{p,\infty}(\mathbb{R}_+^2)$ implies the boundedness of its adjoint $(T^+)^*$ from $L^{p',1}$ to $L^{p',q'}_a$ and we easily check that $$(T^+)^*g(u+iv) = v^{\beta-a}\int_{\mathbb{R}^2_+}\frac{g(x+iy)}{|(x-u)+i(y+v)|^{1+\gamma}}y^{\alpha}\mathrm{d}y.$$ The right inequality is then obtained following the lines of the last part of the proof of the lemma just above. The proof is complete.
\end{proof}
\begin{lemma}
\label{b4}
Let $1<p<\infty$. Assume that $T^+_{\alpha,\beta,\gamma}$ is bounded from $L^{p,1}(\mathbb{R}^2_+)$ to $L^{p,\infty}(\mathbb{R}^2_+).$ Then the parameters satisfy $$\gamma = \alpha+\beta\quad\mbox{and}\quad \alpha,\beta>0.$$ 
\end{lemma}
\begin{proof}
Note that the boundedness of $T^+$ from $L^{p,1}(\mathbb{R}^2_+)$ to $L^{p,\infty}(\mathbb{R}^2_+)$ is equivalent to the boundedness of the adjoint operator $(T^+)^*$ from $L^{p,1}(\mathbb{R}^2_+)$ to $L^{p,\infty}(\mathbb{R}^2_+).$ Here
$$(T^+)^*g(u+iv) = v^{\beta}\int_{\mathbb{R}^2_+}\frac{g(x+iy)}{|(x-u)+i(y+v)|^{1+\gamma}}y^{\alpha}\mathrm{d}y.$$ The whole proof follows the lines of the proof of Lemma \ref{b2}.
\end{proof}
We end this section with the following lemma.
\begin{lemma}
\label{b5}
Let $1<p<\infty$. Assume that $T^+_{\alpha,\beta,\gamma}$ is bounded on $L^{p,1}(\mathbb{R}^2_+)$. Then the parameters satisfy $$\gamma = \alpha+\beta+1\quad\mbox{and}\quad \alpha>-1,\beta>0.$$ 
\end{lemma}
\begin{proof} Assume that $T^+_{\alpha,\beta,\gamma}$ is bounded on $L^{p,1}(\mathbb{R}^2_+)$. Then that $\gamma = \alpha+\beta+1$ and $\alpha>0$ follows as in the first part of the proof of Lemma \ref{b2}. The boundedness of $T^+_{\alpha,\beta,\gamma}$ on $L^{p,1}(\mathbb{R}^2_+)$ is equivalent to the bounded of the operator 
$$(T^+)^*g(u+iv) = v^{\beta}\int_{\mathbb{R}^2_+}\frac{g(x+iy)}{|(x-u)+i(y+v)|^{1+\gamma}}y^{\alpha}\mathrm{d}y$$ on $L^{p',\infty}(\mathbb{R}^2_+)$. This as in the proof of Lemma \ref{b3} gives that $\beta>0$.
\end{proof}
\subsubsection{Proof of the results on operators $T^+$}
The sufficient part in Theorem \ref{thm:main2}, Theorem \ref{thm:main3}, Theorem \ref{thm:main4}, Theorem \ref{thm:main5} and Theorem \ref{thm:main6} follows from Lemma \ref{b11}. The necessity of the conditions in first theorem follows from Lemma \ref{b2}. In Theorem \ref{thm:main3} and Theorem \ref{thm:main4}, the necessity of the given conditions is a consequence of Lemma \ref{b3}, while in  Theorem \ref{thm:main5} and Theorem \ref{thm:main6} it follows from Lemma \ref{b4} and Lemma \ref{b5} respectively. Theorem \ref{thm:main7} is just the dual version of Theorem \ref{thm:main6} and so, its proof also follows from Lemma \ref{b1} and Lemma \ref{b5}.

Let us prove Theorem \ref{thm:main8} and Theorem \ref{thm:main9}.
\begin{proof}[Proof of Theorem \ref{thm:main8}]
First assume that $\alpha>0$ and $\beta>-1$, and $\gamma=\alpha+\beta+1$. Then it follows from Lemma \ref{lem:integkernel} that for any $z=x+iy\in \mathbb{R}_+^2$, the integral $$I(z)
:=\int_{\mathbb{R}_+^2}\frac{v^\beta}{|(x-u)+i(y+v)|^{1+\gamma}}\mathrm{d}u\mathrm{d}v$$
is convergent and $$I(z)=B(\frac{1}{2},\frac{\gamma}{2})B(\beta+1,\alpha)y^{-\gamma+\beta+1}.$$
Hence, as $\gamma=\alpha+\beta+1$, we obtain
$$\sup_{x+iy\in \mathbb{R}_+^2}\,y^\alpha \int_{\mathbb{R}_+^2}\frac{v^\beta}{|(x-u)+i(y+v)|^{1+\gamma}}\mathrm{d}u\mathrm{d}v=B(\frac{1}{2},\frac{\gamma}{2})B(\beta+1,\alpha)<\infty.$$
Conversely, let us suppose that 
$$\sup_{x+iy\in \mathbb{R}_+^2}\,y^\alpha \int_{\mathbb{R}_+^2}\frac{v^\beta}{|(x-u)+i(y+v)|^{1+\gamma}}\mathrm{d}u\mathrm{d}v<\infty.$$
Then in particular, we have $$\int_{\mathbb{R}_+^2}\frac{v^\beta}{|-u+i(1+v)|^{1+\gamma}}\mathrm{d}u\mathrm{d}v<\infty,$$
and by Lemma \ref{lem:integkernel}, this implies that $\beta>-1$ and $\gamma>\beta+1$. Therefore, for any $z=x+iy\in \mathbb{R}_+^2$, we have from the same lemma that
$$B(\frac{1}{2},\frac{\gamma}{2})B(\beta+1,\gamma-\beta-1)y^{-\gamma+\alpha+\beta+1}=y^\alpha \int_{\mathbb{R}_+^2}\frac{v^\beta}{|(x-u)+i(y+v)|^{1+\gamma}}\mathrm{d}u\mathrm{d}v<\infty$$ and this is possible only if $\gamma=\alpha+\beta+1$. Consequently, $\gamma>\beta+1$ is equivalent to $\alpha>0$. The proof is complete.
\end{proof}
We next prove Theorem \ref{thm:main9}.
\begin{proof}[Proof of Theorem \ref{thm:main9}]
Note that the kernel of the operator $T^+$ with respect to the measure $v^a\mathrm{d}u\mathrm{d}v$ is $K(x+iy,u+iv)=\frac{y^\alpha v^{\beta-a}}{|(x-u)+i(y+v)|^{1+\gamma}}$. It follows from Lemma \ref{lem:okikiolulimitcase} that the boundedness of $T^+$ on $L_a^1(\mathbb{R}_+^2)$ is equivalent to the following condition
\begin{equation}\label{eq:equivL1a}
\sup_{u+iv\in \mathbb{R}_+^2}\int_{\mathbb{R}_+^2}K(x+iy,u+iv)y^a\mathrm{d}x\mathrm{d}y.
\end{equation}
Therefore, we only have to prove that (\ref{eq:equivL1a}) is equivalent to $\gamma=\alpha+\beta+1$ and $-\alpha<a+1<\beta+1$. This is handled as in the proof of Theorem \ref{thm:main9}. Lemma \ref{lem:integkernel} gives us that 
\Beas \|T^+\|_{L_a^1\rightarrow L_a^1} &=& \sup_{u+iv\in \mathbb{R}_+^2}\int_{\mathbb{R}_+^2}K(x+iy,u+iv)y^a\mathrm{d}x\mathrm{d}y\\ &=& \sup_{u+iv\in \mathbb{R}_+^2}\int_{\mathbb{R}_+^2}\frac{y^\alpha v^{\beta-a}}{|(x-u)+i(y+v)|^{1+\gamma}}y^a\mathrm{d}x\mathrm{d}y\\ &=& B(\frac{1}{2},\frac{\gamma}{2})B(\beta-a,\alpha+a+1).
\Eeas
The proof is complete.
\end{proof}
\section{Boundedness of a family of Bergman-type operators}
It is proved in \cite{Sehba} that the following holds.
\begin{proposition}\label{prop:main8} The operator $T_{\a,\ba,\ga}$ is bounded on
$L^{\infty}(\mathbb{R}_+^2)$ if and only if $\a >0$,  $\ba
> -1$ and $\ga=\a + \ba + 1$. 
\end{proposition}
The proof of the dual version of this result follows the same way.
\begin{corollary}\label{cor:main9} Let $a>-1$. Then the operator $T_{\a,\ba,\ga}$ is bounded on
$L_a^{1}(\mathbb{R}_+^2)$ if and only if $\ga=\a + \ba + 1$ and $-\alpha<a+1<\beta+1$. 
\end{corollary}
Taking $\alpha=0$ and $\gamma=\beta+1$, we obtain the following application.
%\subsubsection{The case $1<p,q<\infty$}
\begin{corollary}\label{cor:main91} Let $a,\beta>-1$. Then the operator $P_{\beta}$ is a bounded projection from
$L_a^{1}(\mathbb{R}_+^2)$ into $A_a^{1}(\mathbb{R}_+^2)$  if and only if $a<\beta$. 
\end{corollary}
We have the following result.
\begin{corollary}\label{cor:main2}
Suppose $a,b>-1$, $1< p<\infty$, and $1< q\le r< \infty$. Then 

the operator $T=T_{\alpha,\beta,\gamma}$ is bounded from
$L_{a}^{p,q}(\mathbb{R}_+^2)$ to $L_{b}^{p,r}(\mathbb{R}_+^2)$ if and only if 
(\ref{eq:relationalphabetagamma}) and (\ref{eq:condineq1}) hold.
\end{corollary}
\begin{proof}
The sufficiency of the conditions (\ref{eq:relationalphabetagamma}) and (\ref{eq:condineq1}) follows from Theorem \ref{thm:main2}. Thus we only have to check the necessity of these conditions. Assume that $T$ is bounded from
$L_{a}^{p,q}(\mathbb{R}_+^2)$ to $L_{b}^{p,r}(\mathbb{R}_+^2)$. That (\ref{eq:relationalphabetagamma}) holds follows as at the beginning of the proof of Lemma \ref{b1}. To prove the other condition,  let us fix $\zeta=s+it\in \mathbb{R}_+^2$ and define the function $f$ by $f(x+iy):=\frac{(\zeta-x+iy)^{1+\gamma}}{|\zeta-x+iy|^{1+\gamma}}\chi_{[-\frac{1}{4},\frac{1}{4}]}(x)\chi_{[1,2]}(y)$. Then we have for any $x+iy\in \mathbb{R}_+^2$, $$Tf(x+iy)=y^\alpha \int_{\mathbb{R}_+^2}\frac{(\zeta-u+iv)^{1+\gamma}}{|\zeta-u+iv|^{1+\gamma}}\frac{v^\beta}{[(x-u)+i(y+v)]^{1+\gamma}}\chi_{[-\frac{1}{4},\frac{1}{4}]}(u)\chi_{[1,2]}(v)\mathrm{d}u\mathrm{d}v.$$
Taking in particular $x+iy=\zeta$, we obtain 
$$Tf(s+it)=t^\alpha \int_{1}^1\int_{-\frac{1}{4}}^{\frac{1}{4}}\frac{1}{|s-u+i(t+v)|^{1+\gamma}}v^\beta\mathrm{d}u\mathrm{d}v.$$
The remaining of the proof follows the lines of the proof of Lemma \ref{b1}.
\end{proof}
This leads to the following for the Bergman projection.
\begin{corollary}\label{cor:main21}
Suppose $a,b,\beta>-1$, $1< p<\infty$, and $1< q\le r< \infty$. Then 

the operator $P_{\beta}$ is a bounded from
$L_{a}^{p,q}(\mathbb{R}_+^2)$ into $A_{b}^{p,r}(\mathbb{R}_+^2)$ if and only if 
$a+1<q(\beta+1)$ and $\frac{a+1}{q}=\frac{b+1}{r}$.
\end{corollary}
Using the same idea as above, we obtain the following.
\begin{corollary}\label{cor:main3}
Suppose $a>-1$, $1< p<\infty$ and $1<q<\infty$. Then the
 operator $T$ is bounded from
$L_{a}^{p,q}(\mathbb{R}_+^2)$ to $L^{p,\infty}(\mathbb{R}_+^2)$ if and only if 
(\ref{eq:main31}) and (\ref{eq:main32}) hold.
\end{corollary}

Taking $1=q<r<\infty$, we obtain also the following.
\begin{corollary}\label{cor:main4}
Let $1< p,r<\infty$ and let $b>-1$. Then the
 operator $T$ is bounded from
$L^{p,1}(\mathbb{R}_+^2)$ into $L_b^{p,r}(\mathbb{R}_+^2)$
if and only if (\ref{eq:main41}) and (\ref{eq:main42}) hold.
\end{corollary}
In particular, we have the following.
\begin{corollary}\label{cor:main41}
Let $1< p,r<\infty$ and let $b,\beta>-1$. Then the
 operator $P_\beta$ is a bounded projection from
$L^{p,1}(\mathbb{R}_+^2)$ into $A_b^{p,r}(\mathbb{R}_+^2)$
if and only if $\beta>0$ and $r=b+1$.
\end{corollary}
In the limit case $q=1$ and $r=\infty$, we have the following.
\begin{corollary}\label{cor:main5}
Let $1< p<\infty$. Then the
operator $T$ is bounded from
$L^{p,1}(\mathbb{R}_+^2)$ into $L^{p,\infty}(\mathbb{R}_+^2)$
if and only if (\ref{eq:main51}) and (\ref{eq:main52}) hold.
\end{corollary}
We also obtain the following.
\begin{corollary}\label{cor:main6}
Let $1< p<\infty$. Then the
operator $T$ is bounded on
$L^{p,1}(\mathbb{R}_+^2)$ if and only if (\ref{eq:main61}) and (\ref{eq:main62}) hold.
\end{corollary}
Consequently, taking again $\alpha=0$ and $\gamma=\beta+1$, we obtain the following.
\begin{corollary}\label{cor:main61}
Let $1< p<\infty$, and $\beta>-1$. Then the
operator $P_\beta$ is a bounded projection from
$L^{p,1}(\mathbb{R}_+^2)$ into $A^{p,1}(\mathbb{R}_+^2)$ if and only if $\beta>0$.
\end{corollary}
%\subsubsection{The case $q=r=\infty$}
%Taking $q=r=\infty$, we obtain the following.
The dual version of Corollary \ref{cor:main6} is the following.
\begin{corollary}\label{cor:main7}
Let $1< p<\infty$. Then the
 operator $T$ is bounded on
$L^{p,\infty}(\mathbb{R}_+^2)$
if and only if (\ref{eq:main61}) and (\ref{eq:main72}) hold.
\end{corollary}

\bibliographystyle{plain}

\end{document}